\documentclass[11pt,a4paper]{amsart}

\usepackage[T1]{fontenc}
\usepackage[utf8]{inputenc}
\usepackage[english]{babel}
\usepackage{amsmath,amssymb,mathtools}
\usepackage[a4paper,margin=2.7cm]{geometry}
\usepackage{enumitem}
\usepackage{microtype}
\usepackage[colorlinks=true,
            linkcolor=blue,
            citecolor=blue,
            urlcolor=blue,
            linktoc=all]{hyperref}

\newtheorem{theorem}{Theorem}[section]
\newtheorem{lemma}[theorem]{Lemma}
\newtheorem{corollary}[theorem]{Corollary}
\theoremstyle{definition}

\theoremstyle{remark}
\newtheorem{remark}[theorem]{Remark}

\newcommand{\cl}[1]{\overline{#1}}
\newcommand{\cU}{\mathcal U}
\newcommand{\cV}{\mathcal V}
\newcommand{\cC}{\mathcal C}
\newcommand{\sqL}{\operatorname{sqL}}

\title[Cardinality Bounds for Hausdorff SDL Spaces]
{Cardinality Bounds for Hausdorff SDL Spaces}
\author{Gabriel Fernandes}
\address{Institute of Mathematics and Computer Sciences, University of
São Paulo (ICMC--USP), Avenida Trabalhador São-carlense, 400, Centro,
São Carlos, SP 13566-590, Brazil}
\email{fernandes@icmc.usp.br}
\author{João Marcelo Maciel Messias}
\address{Institute of Mathematics and Statistics, University of São
Paulo (IME--USP), Rua do Matão, 1010, Cidade Universitária, 05508-090
São Paulo, SP, Brazil}
\email{jommessias@usp.br}
\date{\today}

\begin{document}

\begin{abstract}
We establish the cardinal inequality
\(|X|\leq 2^{t(X)H\psi(X)}\) for every Hausdorff SDL space \(X\), where
\(t(X)\) and \(H\psi(X)\) denote the tightness and the Hausdorff
pseudocharacter of \(X\), respectively. Since both invariants are
bounded by \(\chi(X)\), this yields \(|X|\leq 2^{\chi(X)}\). As a
consequence, every first-countable Hausdorff strongly
cellular--Lindelöf space has cardinality at most the continuum. These
results answer Questions~2.1 and~2.2 of Bella and Spadaro. An
intermediate result is a uniform bounded-decomposition property for
SDL spaces; in particular, their strict quasi--Lindelöf number
satisfies \(\sqL(X)\leq t(X)\). 
\end{abstract}

\maketitle

\section{Introduction}

Bella and Spadaro~\cite{BellaSpadaroSDL} call a space \(X\) an
\emph{SDL space} if the closure of every strongly discrete subset of
\(X\) is Lindelöf. They proved, among other results, that every
Urysohn SDL space \(X\) satisfies
\[
 |X|\leq 2^{\chi(X)}.
\]
They then asked whether the Urysohn assumption can be weakened to the
Hausdorff separation axiom and whether, as a related consequence,
every first-countable Hausdorff strongly cellular--Lindelöf space has
cardinality at most the continuum; see
\cite[Questions 2.1 and 2.2]{BellaSpadaroSDL}.

Bounded decompositions of open covers have appeared before in this
context. Arhangel'ski\u{\i} introduced strictly quasi--Lindelöf spaces
and proved that every first-countable Hausdorff space with this
property has cardinality at most the continuum
\cite[Corollary 22]{ArhangelskiiGeneric}. Bella and Spadaro later
introduced the strict quasi--Lindelöf number \(\sqL(X)\) and asked
whether every Hausdorff space satisfies
\( |X|\leq 2^{\sqL(X)\chi(X)} \), even in the strictly
quasi--Lindelöf case; see
\cite[Questions 3.3 and 3.4]{BellaSpadaroCommon}. Their
\cite[Lemma 11]{BellaSpadaroCommon} also isolates the corresponding
bounded-decomposition property for initially \(\kappa\)-compact
spaces. In related work, Bella proved that first-countable Hausdorff
strongly cellular--Lindelöf spaces are weakly Lindelöf with respect to
closed sets, but noted that the argument did not establish strict
quasi--Lindelöfness even for cellular-compact spaces
\cite[Section 2]{BellaCellularCompact}.

We answer the two questions about SDL spaces affirmatively. In fact,
we replace the character in the exponent by the product of the
tightness and the Hausdorff pseudocharacter:
\[
 |X|\leq 2^{t(X)H\psi(X)}.
\]
The key new step is a covering lemma showing that, whenever
\(\kappa\geq t(X)\), every SDL space satisfies the relevant covering
property for decompositions into at most \(\kappa\) pieces. At
\(\kappa=t(X)\), this gives \(\sqL(X)\leq t(X)\); the uniform statement
for all \(\kappa\geq t(X)\) is what is used in the cardinality proof.

Bella, Carlson, and Spadaro proved
\( |X|\leq 2^{pwL_c(X)H\psi(X)} \) for Hausdorff spaces
\cite{BellaCarlsonSpadaroHpsi}. After the covering lemma is
established, our proof adapts their Hausdorff-pseudocharacter operator
and elementary-submodel argument. The unrestricted invariant
\(pwL_c(X)\) is replaced by the bounded-decomposition property at
\(\kappa=t(X)H\psi(X)\).

\section{Preliminaries}

All cardinal invariants in this paper are taken to be infinite
cardinals. We use the standard notation for cardinal functions in
topology; see \cite{Engelking,Juhasz}. The cardinal functions needed
below are defined explicitly as follows.

For \(x\in X\), the \emph{tightness at \(x\)}, denoted by \(t(x,X)\),
is the least infinite cardinal \(\kappa\) such that, whenever
\(A\subseteq X\) and \(x\in\cl A\), there is a set \(B\subseteq A\),
with \(|B|\leq\kappa\), such that \(x\in\cl B\). The tightness of the
space is \(t(X)=\sup\{t(x,X):x\in X\}\).

For \(x\in X\), the \emph{character} \(\chi(x,X)\) is the least infinite
cardinal \(\kappa\) for which \(x\) has a local base of cardinality at
most \(\kappa\). The character of the space is
\(\chi(X)=\sup\{\chi(x,X):x\in X\}\).

For \(x\in X\), the \emph{pseudocharacter} \(\psi(x,X)\) is the least
infinite cardinal \(\kappa\) for which there is a family \(\mathcal U_x\)
of open neighborhoods of \(x\), with
\(|\mathcal U_x|\leq\kappa\), such that
\(\{x\}=\bigcap\mathcal U_x\). The pseudocharacter of \(X\) is
\(\psi(X)=\sup\{\psi(x,X):x\in X\}\).
For comparison, the \emph{closed pseudocharacter} \(\psi_c(x,X)\) is
defined in the same way, with
\(\{x\}=\bigcap\{\cl U:U\in\mathcal U_x\}\), and
\(\psi_c(X)=\sup\{\psi_c(x,X):x\in X\}\).

We write \(\mathfrak c=2^\omega\) for the cardinality of the continuum.
Unless explicitly stated otherwise, closures are taken in the ambient
space. As usual, \([A]^{\leq\kappa}\) denotes the family of all subsets
of \(A\) of cardinality at most \(\kappa\).

A subset \(D\) of a space \(X\) is \emph{strongly discrete} if it has
a pairwise disjoint open expansion; that is, there is a family
\(\{W_d:d\in D\}\) of pairwise disjoint open subsets of \(X\) such
that \(d\in W_d\) for every \(d\in D\). The space \(X\) is an
\emph{SDL space} if \(\cl D\) is Lindelöf for every strongly discrete
subset \(D\) of \(X\)~\cite{BellaSpadaroSDL}.

A family of nonempty open sets is \emph{cellular} if its members are
pairwise disjoint. A space \(X\) is \emph{strongly
cellular--Lindelöf} if, for every cellular family \(\cU\) in \(X\),
there is a closed Lindelöf subspace \(L\) of \(X\) meeting every member
of \(\cU\)~\cite{BellaSpadaroSDL}.

A family \(\mathcal P_x\) of nonempty open subsets of \(X\) is a
\emph{local \(\pi\)-base at \(x\)} if every open neighborhood of
\(x\) contains a member of \(\mathcal P_x\). It is a
\emph{disjoint local \(\pi\)-base} if its members are pairwise
disjoint. Notice that members of a local \(\pi\)-base at \(x\) need
not contain \(x\).

Let \(X\) be Hausdorff. The \emph{Hausdorff pseudocharacter}
\(H\psi(X)\) is the least infinite cardinal \(\kappa\) for which one
can assign to each \(x\in X\) a family \(\mathcal N_x\) of open
neighborhoods of \(x\), with \(|\mathcal N_x|\leq\kappa\), such that
whenever \(x\neq y\), there are \(U\in\mathcal N_x\) and
\(V\in\mathcal N_y\) satisfying \(U\cap V=\varnothing\); see
\cite{BellaCarlsonSpadaroHpsi,Hodel}. By adjoining \(X\) and taking
finite intersections, without increasing the cardinalities, we may
and shall assume that \(X\in\mathcal N_x\) and that every
\(\mathcal N_x\) is closed under finite intersections.

Thus, for every Hausdorff space,
\[
 \psi(X)\leq\psi_c(X)\leq H\psi(X)\leq\chi(X).
\]

For comparison with the covering lemma below, recall the
\emph{piecewise weak Lindelöf degree for closed sets}, denoted by
\(pwL_c(X)\), introduced in
\cite[Definition 5]{BellaSpadaroCommon}. It is the least infinite
cardinal \(\kappa\) such that,
for every closed set \(F\subseteq X\), every open cover \(\mathcal U\)
of \(F\), and every decomposition
\(\mathcal U=\bigcup_{i\in I}\mathcal U_i\),
there are families
\(\mathcal V_i\in[\mathcal U_i]^{\leq\kappa}\), for \(i\in I\),
satisfying
\[
 F\subseteq\bigcup_{i\in I}\cl{\bigcup\mathcal V_i}.
\]
No restriction is imposed on the size of the index set \(I\) in the
definition of \(pwL_c(X)\).

The \emph{strict quasi--Lindelöf number} \(\sqL(X)\), introduced by
Bella and Spadaro~\cite[Section 3]{BellaSpadaroCommon}, is the least
infinite cardinal \(\kappa\) such that the same conclusion holds
whenever the decomposition has at most \(\kappa\) pieces: for every
closed set \(F\subseteq X\), every open cover \(\mathcal U\) of
\(F\), and every decomposition
\[
 \mathcal U=\bigcup_{i\in I}\mathcal U_i,
 \qquad |I|\leq\kappa,
\]
there are families
\(\mathcal V_i\in[\mathcal U_i]^{\leq\kappa}\) such that
\[
 F\subseteq\bigcup_{i\in I}\cl{\bigcup\mathcal V_i}.
\]
Thus \(\sqL(X)\leq pwL_c(X)\). The equality \(\sqL(X)=\omega\) is
the strictly quasi--Lindelöf property introduced by
Arhangel'ski\u{\i}~\cite[Section 14]{ArhangelskiiGeneric}.

We shall use \(t(X)\leq\chi(X)\) and
\(H\psi(X)\leq\chi(X)\).
For completeness, the first inequality follows by choosing, from each
member of a local base at a point \(x\in\cl A\), one point of \(A\).
The second follows by refining disjoint Hausdorff neighborhoods with
local bases of size at most \(\chi(X)\).

We now make explicit the assignment that will be used below. For every
\(x\in X\), choose a family \(\mathcal N_x\) satisfying the requirements
above. If \(\tau_X\) denotes the topology of \(X\), define
\[
 \Phi:X\longrightarrow[\tau_X]^{\leq H\psi(X)},
 \qquad \Phi(x)=\mathcal N_x.
\]
Thus, for every \(x\in X\), the members of \(\Phi(x)\) are open
neighborhoods of \(x\), \(|\Phi(x)|\leq H\psi(X)\),
\(X\in\Phi(x)\), and \(\Phi(x)\) is closed under finite
intersections. Moreover, if \(x,y\in X\) are distinct, then there are
\(U\in\Phi(x)\) and \(V\in\Phi(y)\) such that \(U\cap V=\varnothing\).

Following \cite[Section 2]{BellaCarlsonSpadaroHpsi}, for
\(A\subseteq X\), define
\[
 c_\Phi(A)=
 \{x\in X:(\forall U\in\Phi(x))\ U\cap A\neq\varnothing\}.
\]
Since every member of \(\Phi(x)\) is a neighborhood of \(x\), we have
\(\cl A\subseteq c_\Phi(A)\). Moreover, if \(c_\Phi(A)=A\), then
\(A\) is closed: for each \(x\notin A\), some
\(U\in\Phi(x)\) is disjoint from \(A\).

A transfinite sequence
\(\langle x_\alpha:\alpha<\lambda\rangle\) in \(X\) is
\emph{free} if, for every \(\beta<\lambda\),
\[
 \cl{\{x_\alpha:\alpha<\beta\}}
 \cap
 \cl{\{x_\alpha:\beta\leq\alpha<\lambda\}}
 =\varnothing.
\]

We record the standard relationship between free sequences,
Lindelöfness, and tightness. A proof is included so that no separation
axiom is tacitly used.

\begin{lemma}\label{lem:free}
Let \(L\) be a Lindelöf space and let \(\kappa\) be an infinite
cardinal. If \(t(L)\leq\kappa\), then \(L\) contains no free sequence
of length \(\kappa^+\).
\end{lemma}

\begin{proof}
Suppose that \(\langle x_\alpha:\alpha<\kappa^+\rangle\) is free in
\(L\). Its terms are distinct, since a repeated term would belong both
to the closure of a suitable initial segment and to the closure of the
corresponding tail. Hence
\(D=\{x_\alpha:\alpha<\kappa^+\}\) has cardinality \(\kappa^+\).

The set \(D\) has a complete accumulation point \(p\) in \(L\).
Indeed, otherwise every point of \(L\) would have an open neighborhood
meeting \(D\) in fewer than \(\kappa^+\), hence in at most \(\kappa\),
points. Lindelöfness would then give a countable cover of \(L\) by
such neighborhoods, forcing \(|D|\leq\kappa\cdot\omega=\kappa\), a
contradiction.

Since \(p\in\cl D\) and \(t(L)\leq\kappa\), there is a set
\(A\in[D]^{\leq\kappa}\) such that \(p\in\cl A\), where closures in
this paragraph are taken in \(L\). By the regularity of \(\kappa^+\),
there is \(\beta<\kappa^+\) such that
\(A\subseteq\{x_\alpha:\alpha<\beta\}\). Thus \(p\) belongs to the
closure of this initial segment. Since the initial segment has size at
most \(\kappa\), every neighborhood of the complete accumulation point
\(p\) also meets the tail
\(\{x_\alpha:\beta\leq\alpha<\kappa^+\}\). Hence \(p\) belongs to the
closure of the tail as well, contrary to freeness.
\end{proof}

\section{A covering lemma for bounded decompositions}

The following lemma is the key topological ingredient. It establishes
for SDL spaces, uniformly at every \(\kappa\geq t(X)\), the
bounded-decomposition property underlying the strict
quasi--Lindelöf number. In contrast with the definition of the
piecewise weak Lindelöf degree for closed sets, the decomposition is
required to have at most \(\kappa\) pieces.

\begin{lemma}\label{lem:piecewise}
Let \(X\) be an SDL space, let \(\kappa\) be an infinite cardinal with
\(t(X)\leq\kappa\), and let \(F\) be a closed subset of \(X\). Suppose
that \(\cU\) is an open cover of \(F\) and
\[
 \cU=\bigcup_{i\in I}\cU_i,
 \qquad |I|\leq\kappa.
\]
Then, for each \(i\in I\), there is a family
\(\cV_i\in[\cU_i]^{\leq\kappa}\) such that
\[
 F\subseteq\bigcup_{i\in I}\cl{\bigcup\cV_i}.
\]
\end{lemma}

\begin{proof}
Suppose, toward a contradiction, that no such families exist. We
construct recursively, for \(\alpha<\kappa^+\), a point
\(x_\alpha\in F\), an index \(i_\alpha\in I\), and open sets
\[
 x_\alpha\in W_\alpha\subseteq U_\alpha\in\cU_{i_\alpha}
\]
so that, for every \(i\in I\), the sets \(W_\alpha\) with
\(i_\alpha=i\) are pairwise disjoint.

Assume that the construction has been carried out before stage
\(\alpha\). For \(i\in I\), put
\(D_i^\alpha=\{x_\beta:\beta<\alpha\text{ and }i_\beta=i\}\). This set
is strongly discrete, as witnessed by those already chosen sets
\(W_\beta\) for which \(i_\beta=i\). Thus \(\cl{D_i^\alpha}\) is
Lindelöf. Since \(F\) is closed,
\(\cl{D_i^\alpha}\subseteq F\), so there is a countable family
\(\cC_i^\alpha\subseteq\cU\) covering \(\cl{D_i^\alpha}\).

For every \(j\in I\), define
\[
 \cV_j^\alpha=
 \{U_\beta:\beta<\alpha,\ i_\beta=j\}
 \cup
 \bigcup_{\substack{\gamma\leq\alpha\\ i\in I}}
       (\cC_i^\gamma\cap\cU_j).
\]
Because \(\alpha<\kappa^+\), we have \(|\alpha|\leq\kappa\).
Together with \(|I|\leq\kappa\) and the countability of all the
families \(\cC_i^\gamma\), this gives
\(|\cV_j^\alpha|\leq\kappa\) for every \(j\in I\). Our assumption
therefore permits us to choose
\[
 x_\alpha\in
 F\setminus\bigcup_{j\in I}\cl{\bigcup\cV_j^\alpha}.
\]
Choose \(U_\alpha\in\cU\) with \(x_\alpha\in U_\alpha\), and then
choose \(i_\alpha\in I\) such that
\(U_\alpha\in\cU_{i_\alpha}\). Since
\(x_\alpha\notin\cl{\bigcup\cV_{i_\alpha}^\alpha}\), there is an open
set \(O_\alpha\) containing \(x_\alpha\) and disjoint from
\(\bigcup\cV_{i_\alpha}^\alpha\). Set
\begin{itemize}
 \item \(W_\alpha=O_\alpha\cap U_\alpha\);
 \item \(x_\alpha\in W_\alpha\subseteq U_\alpha\);
 \item \(W_\alpha\cap\bigcup\cV_{i_\alpha}^\alpha=\varnothing\).
\end{itemize}
If \(\beta<\alpha\) and \(i_\beta=i_\alpha\), then
\(U_\beta\in\cV_{i_\alpha}^\alpha\), and
\(W_\beta\subseteq U_\beta\). Hence
\(W_\alpha\cap W_\beta=\varnothing\), so the recursion continues.

Since \(|I|\leq\kappa\) and \(\kappa^+\) is regular, some
\(i_*\in I\) occurs \(\kappa^+\) times. Write
\[
 \{\alpha<\kappa^+:i_\alpha=i_*\}
 =\{a_\xi:\xi<\kappa^+\},
\]
where the enumeration is increasing, and put
\(y_\xi=x_{a_\xi}\). The set
\(D=\{y_\xi:\xi<\kappa^+\}\) is strongly discrete, as witnessed by
\(\{W_{a_\xi}:\xi<\kappa^+\}\).

We claim that \(\langle y_\xi:\xi<\kappa^+\rangle\) is free. Fix
\(\xi<\kappa^+\). At stage \(a_\xi\), the open set
\(G_\xi=\bigcup\cC_{i_*}^{a_\xi}\) covers
\(\cl{\{y_\eta:\eta<\xi\}}\). If
\(C\in\cC_{i_*}^{a_\xi}\), then \(C\in\cU_j\) for some \(j\in I\).
For every \(\eta\geq\xi\), the cumulative definition gives
\(C\in\cV_j^{a_\eta}\). By the choice of \(x_{a_\eta}\), we have
\(y_\eta=x_{a_\eta}\notin
\cl{\bigcup\cV_j^{a_\eta}}\), and hence \(y_\eta\notin C\). Thus the
tail \(\{y_\eta:\eta\geq\xi\}\) is contained in the closed set
\(X\setminus G_\xi\), and so its closure is disjoint from \(G_\xi\).
Consequently,
\[
 \cl{\{y_\eta:\eta<\xi\}}
 \cap
 \cl{\{y_\eta:\xi\leq\eta<\kappa^+\}}
 =\varnothing,
\]
which proves the claim.

By the SDL property, \(L=\cl D\) is Lindelöf. Since \(L\) is closed in
\(X\), \(t(L)\leq t(X)\leq\kappa\): indeed, if
\(A\subseteq L\) and \(p\) belongs to the closure of \(A\) in \(L\),
then \(p\in\cl A\) in \(X\), and a set witnessing tightness in \(X\)
also witnesses tightness in \(L\). The sequence remains free in the
subspace \(L\), contradicting Lemma~\ref{lem:free}.
\end{proof}

\begin{corollary}\label{cor:sql}
Every SDL space \(X\) satisfies
\[
 \sqL(X)\leq t(X).
\]
In particular, every countably tight SDL space is strictly
quasi--Lindelöf.
\end{corollary}

\begin{proof}
Apply Lemma~\ref{lem:piecewise} with \(\kappa=t(X)\). When
\(t(X)=\omega\), the resulting property is precisely strict
quasi--Lindelöfness.
\end{proof}

\section{The cardinal inequality and its consequences}

\begin{theorem}\label{thm:main}
If \(X\) is a Hausdorff SDL space, then
\[
 |X|\leq 2^{t(X)H\psi(X)}\leq 2^{\chi(X)}.
\]
\end{theorem}

\begin{proof}
The elementary-submodel part of the argument adapts the proof scheme
of \cite[Section 2]{BellaCarlsonSpadaroHpsi}; the new input is
Lemma~\ref{lem:piecewise}.
Set \(\kappa=t(X)\cdot H\psi(X)\).
Choose a sufficiently large regular cardinal \(\Theta\) and an
elementary submodel \(M\prec H(\Theta)\) such that
\[
 X,\tau_X,\Phi,\kappa\in M,\qquad
 \kappa+1\subseteq M,\qquad
 |M|=2^\kappa,\qquad
 [M]^{\leq\kappa}\subseteq M.
\]
The standard Skolem-hull construction produces such a model; see
\cite{Dow,Kunen}; the relevant cardinal arithmetic is
\((2^\kappa)^\kappa=2^{\kappa\cdot\kappa}=2^\kappa\).

We shall repeatedly use two closure facts. First, every subset of
\(M\) of cardinality at most \(\kappa\) belongs to \(M\). Second, if
\(B\in M\) and \(|B|\leq\kappa\), then \(B\subseteq M\). For the
second fact, if \(B\) is nonempty, elementarity gives a surjection
\(f\in M\) from \(\kappa\) onto \(B\). Since \(\kappa\subseteq M\),
we have \(f(\xi)\in M\) for every \(\xi<\kappa\), and hence
\(B\subseteq M\); the empty case is immediate.

Let \(F=X\cap M\). We first prove that \(c_\Phi(F)=F\), which also
shows that \(F\) is closed in \(X\). Let
\(p\in c_\Phi(F)\). For each \(V\in\mathcal N_p\), choose
\(x_V\in V\cap F\), and set
\(A_p=\{x_V:V\in\mathcal N_p\}\).
Then \(A_p\subseteq M\) and \(|A_p|\leq\kappa\), so \(A_p\in M\).

Fix \(V\in\mathcal N_p\). For every \(W\in\mathcal N_p\), the set
\(T=V\cap W\) belongs to \(\mathcal N_p\), and
\(x_T\in V\cap W\cap A_p\). Hence \(p\in c_\Phi(V\cap A_p)\), and
this holds for every \(V\in\mathcal N_p\). If \(q\neq p\), choose
\(U_q\in\mathcal N_p\) and \(W_q\in\mathcal N_q\) with
\(U_q\cap W_q=\varnothing\). Then \(W_q\) is disjoint from
\(U_q\cap A_p\), so \(q\notin c_\Phi(U_q\cap A_p)\). Thus, for
\[
 \mathcal E_p=\{c_\Phi(V\cap A_p):V\in\mathcal N_p\},
 \qquad
 \bigcap\mathcal E_p=\{p\}.
\]
For each \(V\in\mathcal N_p\), the set \(V\cap A_p\) is a subset of
\(M\) of cardinality at most \(\kappa\), and hence belongs to \(M\).
The set \(c_\Phi(V\cap A_p)\) is definable from \(V\cap A_p\), \(X\),
and \(\Phi\), so it also belongs to \(M\). Therefore
\(\mathcal E_p\subseteq M\) and \(|\mathcal E_p|\leq\kappa\), whence
\(\mathcal E_p\in M\). Since intersection is definable,
\(\{p\}=\bigcap\mathcal E_p\in M\), and therefore \(p\in M\). This
proves \(c_\Phi(F)\subseteq F\); the reverse inclusion is immediate.

We now show that \(X\subseteq M\). Suppose, toward a contradiction,
that \(p\in X\setminus M\), and enumerate
\(\mathcal N_p=\{B_\alpha:\alpha<\kappa\}\), repeating members if
necessary. Since \(p\notin M\), each \(x\in F\) is distinct from
\(p\). Thus the definition of \(H\psi(X)\) yields
\(U_x\in\mathcal N_x\) and \(\alpha(x)<\kappa\) such that
\[
 x\in U_x,
 \qquad
 U_x\cap B_{\alpha(x)}=\varnothing.
\]
Since \(x,\Phi\in M\), we have
\(\mathcal N_x=\Phi(x)\in M\). Moreover,
\(|\mathcal N_x|\leq\kappa\), so the second closure fact gives
\(\mathcal N_x\subseteq M\). In particular, \(U_x\in M\).

The family \(\cU=\{U_x:x\in F\}\)
is an open cover of \(F\). Decompose it into at most \(\kappa\)
pieces by setting
\[
 \cU_\alpha=
 \{U_x:x\in F\text{ and }\alpha(x)=\alpha\},
 \qquad \alpha<\kappa.
\]
Lemma~\ref{lem:piecewise} gives families
\(\cV_\alpha\in[\cU_\alpha]^{\leq\kappa}\) such that
\[
 F\subseteq
 E:=\bigcup_{\alpha<\kappa}\cl{\bigcup\cV_\alpha}.
\]

Each \(\cV_\alpha\) is a subset of \(M\) of cardinality at most
\(\kappa\), so \(\cV_\alpha\in M\). Since
\(\alpha,\cV_\alpha\in M\) for each \(\alpha<\kappa\), every ordered
pair \(\langle\alpha,\cV_\alpha\rangle\) belongs to \(M\). Hence the
graph of \(\langle\cV_\alpha:\alpha<\kappa\rangle\) is a subset of
\(M\) of cardinality at most \(\kappa\), and thus belongs to \(M\).
The set \(E\), being definable from this sequence, \(X\), and
\(\tau_X\), also belongs to \(M\).

If \(X\nsubseteq E\), elementarity would produce a point of
\((X\setminus E)\cap M\), contradicting
\(F=X\cap M\subseteq E\). Hence \(X\subseteq E\). On the other hand,
for every \(\alpha<\kappa\), the open neighborhood \(B_\alpha\) of
\(p\) is disjoint from \(\bigcup\cV_\alpha\). Hence
\(p\notin\cl{\bigcup\cV_\alpha}\) for every \(\alpha<\kappa\), and
therefore \(p\notin E\), a contradiction.

We conclude that \(X\subseteq M\), and consequently
\[
 |X|\leq|M|=2^\kappa=2^{t(X)H\psi(X)}.
\]
Finally, \(t(X),H\psi(X)\leq\chi(X)\), and multiplication of infinite
cardinals is their maximum. Hence
\(t(X)\cdot H\psi(X)\leq\chi(X)\), which proves the second inequality.
\end{proof}

\begin{corollary}[Answer to Question 2.1]\label{cor:q21}
Every Hausdorff SDL space \(X\) satisfies
\[
 |X|\leq 2^{\chi(X)}.
\]
\end{corollary}

For completeness, we include the reduction of Question~2.2 to
Question~2.1. This reduction already appears in
\cite[Lemmas 1 and 2]{BellaCellularCompact} and also in
\cite[Lemmas 3 and 4]{BellaSpadaroSDL}.

\begin{lemma}\label{lem:first-countable-sdl}
Every first-countable Hausdorff space has a disjoint local
\(\pi\)-base at each point. Consequently, every first-countable
Hausdorff strongly cellular--Lindelöf space is an SDL space.
\end{lemma}

\begin{proof}
Fix \(x\in X\). If \(x\) is isolated, then
\(\{\{x\}\}\) is a disjoint local \(\pi\)-base at \(x\). Suppose that
\(x\) is not isolated, and fix a decreasing local base
\(\{U_n:n<\omega\}\) at \(x\).

We recursively construct integers \(k_0<k_1<\cdots\) and nonempty
open sets \(P_n\) such that
\[
 P_n\subseteq U_{k_n}
 \qquad\text{and}\qquad
 P_n\cap U_{k_{n+1}}=\varnothing.
\]
Set \(k_0=0\). Once \(k_n\) has been chosen, take
\(y_n\in U_{k_n}\setminus\{x\}\). By the Hausdorff property, there
are disjoint open sets \(O_n\) and \(Q_n\) with
\(x\in O_n\) and \(y_n\in Q_n\). Put
\(P_n=Q_n\cap U_{k_n}\), and choose \(k_{n+1}>k_n\) so that
\(U_{k_{n+1}}\subseteq O_n\). Then \(P_n\) is nonempty and has the
required properties.

If \(m>n\), then
\(P_m\subseteq U_{k_m}\subseteq U_{k_{n+1}}\), so
\(P_m\cap P_n=\varnothing\). Moreover, the strictly increasing
sequence \((k_n)_{n<\omega}\) is cofinal in \(\omega\). Given a
neighborhood \(G\) of \(x\), choose \(r<\omega\) such that
\(U_r\subseteq G\), and then choose \(n\) with \(k_n\geq r\). It
follows that \(P_n\subseteq U_{k_n}\subseteq U_r\subseteq G\).
Therefore \(\{P_n:n<\omega\}\) is a disjoint local \(\pi\)-base at
\(x\).

Now suppose that \(X\) is also strongly cellular--Lindelöf. Let
\(D\subseteq X\) be strongly discrete, witnessed by a pairwise
disjoint open expansion \(\{W_d:d\in D\}\). For each \(d\in D\),
choose a disjoint local \(\pi\)-base \(\mathcal P_d\) at \(d\) all of
whose members are contained in \(W_d\). This can be done by starting
the preceding construction with a local base whose first member is
contained in \(W_d\). The family
\(\mathcal P=\bigcup_{d\in D}\mathcal P_d\) is cellular. Hence there
is a closed Lindelöf subspace \(L\) of \(X\) meeting every member of
\(\mathcal P\).

Every neighborhood of \(d\in D\) contains some
\(P\in\mathcal P_d\), and this \(P\) meets \(L\). Thus
\(d\in\cl L=L\). It follows that \(D\subseteq L\), and therefore
\(\cl D\subseteq L\). Since \(\cl D\) is closed in the Lindelöf space
\(L\), it is Lindelöf. Thus \(X\) is an SDL space.
\end{proof}

\begin{corollary}[Answer to Question 2.2]\label{cor:q22}
If \(X\) is first-countable, Hausdorff, and strongly
cellular--Lindelöf, then
\[
 |X|\leq\mathfrak c.
\]
\end{corollary}

\begin{proof}
By Lemma~\ref{lem:first-countable-sdl}, the space \(X\) is SDL. Since
\(\chi(X)=\omega\), Theorem~\ref{thm:main} gives
\(|X|\leq 2^\omega=\mathfrak c\).
Alternatively, Corollary~\ref{cor:sql} and
\cite[Corollary 22]{ArhangelskiiGeneric} give the same conclusion.
\end{proof}

\begin{remark}
Lemma~\ref{lem:piecewise} does not assert the unrestricted inequality
\(pwL_c(X)\leq t(X)\). At \(\kappa=t(X)\), it yields the weaker
inequality \(\sqL(X)\leq t(X)\), but the lemma says more: it proves the
bounded-decomposition property directly at every
\(\kappa\geq t(X)\). We use it at
\(\kappa=t(X)H\psi(X)\), because the pieces in the proof of
Theorem~\ref{thm:main} are indexed by the family \(\mathcal N_p\)
associated with a single point \(p\).
\end{remark}

\section*{Acknowledgments}

The results presented in this paper were obtained as part of the
undergraduate research project of João Marcelo Maciel Messias, under
the supervision of Gabriel Fernandes.

João Marcelo Maciel Messias was supported by the São Paulo Research
Foundation (FAPESP), grant 2025/08845-7. Gabriel Fernandes was supported
by FAPESP, grant 2025/09425-1.

\section*{Declaration of AI use}

OpenAI's GPT-5.6 Sol, accessed through Codex, was used for brainstorming
and testing ideas, producing first drafts of proofs, and assisting with
bibliographic searches, proofreading, LaTeX editing, and manuscript
revision. All AI-generated output and suggested references were
reviewed, hand-edited, and revised by the authors, who take full
responsibility for the final manuscript.


\begin{thebibliography}{99}

\bibitem{ArhangelskiiGeneric}
A.~V. Arhangel'ski\u{\i},
\newblock A generic theorem in the theory of cardinal invariants of
topological spaces,
\newblock \emph{Comment. Math. Univ. Carolin.} \textbf{36} (1995),
303--325.

\bibitem{BellaCellularCompact}
A.~Bella,
\newblock On cellular-compact and related spaces,
\newblock \emph{Topology Appl.} \textbf{281} (2020), Article 107203,
\newblock \href{https://doi.org/10.1016/j.topol.2020.107203}
{doi:10.1016/j.topol.2020.107203}.

\bibitem{BellaCarlsonSpadaroHpsi}
A.~Bella, N.~Carlson, and S.~Spadaro,
\newblock Cardinal inequalities involving the Hausdorff
pseudocharacter,
\newblock \emph{Rev. Real Acad. Cienc. Exactas Fis. Nat. Ser. A-Mat.}
\textbf{117} (2023), Article 129,
\newblock \href{https://doi.org/10.1007/s13398-023-01460-4}
{doi:10.1007/s13398-023-01460-4}.

\bibitem{BellaSpadaroCommon}
A.~Bella and S.~Spadaro,
\newblock A common extension of Arhangel'ski\u{\i}'s theorem and the
Hajnal--Juh\'asz inequality,
\newblock \emph{Canad. Math. Bull.} \textbf{63} (2020), 197--203,
\newblock \href{https://doi.org/10.4153/S0008439519000420}
{doi:10.4153/S0008439519000420}.

\bibitem{BellaSpadaroSDL}
A.~Bella and S.~Spadaro,
\newblock Strongly discrete subsets with Lindelöf closures,
\newblock \emph{Topology Proceedings} \textbf{59} (2022), 89--98.

\bibitem{Dow}
A.~Dow,
\newblock An introduction to applications of elementary submodels to
topology,
\newblock \emph{Topology Proceedings} \textbf{13} (1988), 17--72.

\bibitem{Engelking}
R.~Engelking,
\newblock \emph{General Topology},
\newblock revised and completed ed., Heldermann Verlag, Berlin, 1989.

\bibitem{Hodel}
R.~E. Hodel,
\newblock Combinatorial set theory and cardinal function inequalities,
\newblock \emph{Proceedings of the American Mathematical Society}
\textbf{111} (1991), 567--575.

\bibitem{Juhasz}
I.~Juhász,
\newblock \emph{Cardinal Functions in Topology---Ten Years Later},
\newblock Mathematical Centre Tracts, vol.~123, Mathematisch Centrum,
Amsterdam, 1980.

\bibitem{Kunen}
K.~Kunen,
\newblock \emph{Set Theory},
\newblock Studies in Logic, vol.~34, College Publications, London,
2011.

\end{thebibliography}
\end{document}